\documentclass[12pt]{amsart}
\usepackage{amsmath, amssymb, amsthm, amscd}
\usepackage{mathtools} % \mathclap, \adjustlimits
\usepackage{enumitem}
\usepackage[margin=3cm]{geometry}
\usepackage{booktabs}
\usepackage{graphicx}

\theoremstyle{plain}
\newtheorem{theorem}{Theorem}
\newtheorem{lemma}[theorem]{Lemma}

\newtheorem{corollary}[theorem]{Corollary}
\theoremstyle{definition}
\newtheorem{definition}[theorem]{Definition}

\DeclareMathOperator{\HD}{HD}
\DeclareMathOperator{\Lip}{Lip}
\DeclareMathOperator{\HausLip}{HausLip}

\DeclareMathOperator{\Skew}{SkewLoc}
\DeclareMathOperator{\diag}{diag}

\DeclareMathOperator{\dimtop}{dim}
\newcommand{\logplus}{\log^+}
\newcommand{\R}{\mathbb{R}}
\newcommand{\Z}{\mathbb{Z}}
\newcommand{\N}{\mathbb{N}}
\renewcommand{\epsilon}{\varepsilon}
\renewcommand{\mid}{\hspace{5pt minus 5pt}|\hspace{5pt minus 5pt}}

\allowdisplaybreaks %Allows very long equations to be broken across pages

\begin{document}

\title{Inequalities for Entropy, Hausdorff Dimension, and Lipschitz Constants}
\author{Samuel Roth}
\author{Zuzana Roth}
\address{Silesian University in Opava, Na Rybni\v{c}ku 1, 74601 Opava, Czech Republic}
\email{zuzana.roth@math.slu.cz}
\email{samuel.roth@math.slu.cz}
\thanks{Research supported by SGS 16/2016 and by RVO funding for I\v{C}47813059.}
\begin{abstract}
We construct suitable metrics for two classes of topological dynamical systems (linear maps on the torus and non-invertible expansive maps on compact spaces) in order to get a lower bound for topological entropy in terms of the resulting Hausdorff dimensions and Lipschitz constants. This reverses an old inequality of Dai, Zhou, and Gheng and leads to a short proof of a well-known theorem on expansive mappings. It also suggests a new invariant of topological conjugacy for dynamical systems on compact metric spaces. 
\end{abstract}
\maketitle

\section{Introduction}

A common task in topological dynamics is to estimate the entropy of a given system. There are several inequalities available for this purpose; many of them combine notions of \emph{dimension}: Hausdorff dimension, topological dimension, dimension of the Riemannian manifold, etc., together with \emph{expansion}: the norm of the derivative, Lyapunov exponents, volume growth, etc.~\cite{Bo, Ru, Yo}. One of these inequalities, first proved twenty years ago by Dai, Zhou, and Gheng, says that a Lipschitz continuous map on a compact metric space with Lipschitz constant $L\geq1$ has topological entropy no more than the Hausdorff dimension of the space times the logarithm of $L$~\cite{DZG}. An interesting feature of this upper bound for entropy is the possibility to sharpen it by changing the metric -- a better metric may give the map a smaller Lipschitz constant or the space a smaller Hausdorff dimension. We are interested in knowing how sharp this inequality can get. Therefore, for a given compact metrizable space $X$ and continuous map $f:X\to X$, we propose to study the number
\begin{equation}\label{start}
\HausLip(X,f):=\inf_{d\in\mathcal{D}(X)} \HD_d(X)\cdot\logplus\Lip_d(f),
\end{equation}
where the infimum is taken over the set $\mathcal{D}(X)$ of metrics on $X$ compatible with its topology -- for further definitions, see Section~\ref{sec:defs}. We call this number the \emph{HausLip constant} of the system $(X,f)$. It is an invariant of topological conjugacy, bounded from below by the topological entropy. Then the question naturally arises: \emph{Is the HausLip constant equal to the entropy?}

%The philosophy behind this question is that by varying the metric $d$ on $X$ we can hope to recover topological information. This idea has been used before with some success. For example, Szpilrajn discovered that the topological dimension of a compact metrizable space $X$ is given by $\inf_{d\in\mathcal{D}(X)} \HD_d(X)$, the infimum of Hausdorff dimensions over all compatible metrics~\cite{Sz}.

Unfortunately, we find that the HausLip constant is not equal to the entropy in general (or fortunately, if we like new conjugacy invariants). To see this, it suffices to take $X$ the Hilbert cube and $f$ a map with positive but finite entropy. Then for any compatible metric $d$ on $X$ we have $\HD_d(X)=\infty$, since topological dimension gives a lower bound for Hausdorff dimension~\cite{Sz}. At the same time, $\Lip_d(f)>1$, since a map with positive entropy must increase some distances, so that the HausLip constant is clearly infinite.

Nevertheless, in many simple classes of dynamical systems, the HausLip constant is equal to the entropy. In dimension zero, this happens for any closed subshift in a finite-alphabet shift space. In dimension one, this holds for piecewise monotone interval maps. It also holds for $C^\infty$-smooth interval maps. We do not know if it holds for all interval maps. We briefly discuss each of these examples in Section~\ref{sec:defs}, with references to the appropriate literature.

The main contribution of this paper is a positive answer for two more classes of maps: linear maps on the torus -- Section~\ref{sec:torus} -- and one-sided expansive systems -- Section~\ref{sec:expansive}. In each case we give an explicit construction of metrics for which the Hausdorff dimension times the logarithm of the Lipschitz constant is close to the entropy. The metrics we construct on the torus exploit the local product structure in terms of the stable, neutral, and unstable eigendirections of the linear map. For the expansive systems, our metric gives a nearly uniform separation rate to all pairs of points sufficiently close to the diagonal. This is related to Fathi's construction of a ``hyperbolic metric'' for a two-sided expansive homeomorphism~\cite{Fa}.

Our work on expansive mappings leads to the following nice corollary: a compact metric space $X$ admitting a one-sided expansive mapping necessarily has finite topological dimension. This is the one-sided version of a theorem by Ma\v{n}\'{e}~\cite{Ma}; the idea of deriving it from an entropy inequality comes from Fathi~\cite{Fa}.

%%%%%%%%%%%%%%%%%%%%%%%%%%%%%%%%%%%%%

\section{Definitions and Background Results}\label{sec:defs}

Let $X$ be a compact metrizable space and $f:X\to X$ a continuous map. A metric $d:X\times X\to[0,\infty)$ is \emph{compatible} if the topology it induces coincides with the topology of $X$. We write $\mathcal{D}(X)$ for the set of all compatible metrics. The \emph{Lipschitz constant} $\Lip_d(f)$ and the \emph{local skew} $\Skew_d(f)$ (this is a kind of anti-Lipschitz constant) of $f$ with respect to the metric $d$ are the quantities
\begin{equation*}
\Lip_d(f)=\sup_{0<d(x,y)} \frac{d(fx, fy)}{d(x,y)}, \qquad \Skew_d(f)=\adjustlimits\sup_{\epsilon>0} \inf_{0<d(x,y)<\epsilon} \frac{d(fx,fy)}{d(x,y)}.
\end{equation*}
We write $\logplus(y)$ to denote the maximum of $\log(y)$ and $0$.

The diameter of a subset $B\subset X$ under the metric $d$ is $|B|_d=\sup\{d(x,y)\mid x,y\in B\}$. An \emph{$\epsilon$-cover} of $X$ is a collection of sets $B_i$ each of diameter $\leq\epsilon$ whose union equals $X$. Then $\mu^s_\epsilon(X)=\inf \sum_i |B_i|_d^s$ where the infimum is taken over all possible $\epsilon$-covers. As $\epsilon$ decreases, the class of $\epsilon$-covers also decreases, leading to a well-defined limit $\mu^s(X)=\lim_{\epsilon\to0}\mu^s_\epsilon(X)$, the so-called \emph{$s$-dimensional Hausdorff measure}. Then the \emph{Hausdorff dimension} $\HD_d(X)$ of $X$ with respect to the metric $d$ is 
\begin{equation*}
\HD_d(X)=\inf\{s\geq0\mid\mu^s(X)=0\}.
\end{equation*}

There is also the \emph{topological dimension} $\dimtop(X)$ (i.e. Lebesgue covering dimension), which may be defined as the smallest integer $m$ such that for each $\epsilon>0$ there is an $\epsilon$-cover of $X$ by open sets such that each point of $X$ belongs to at most $m+1$ members of the cover; if no such $m$ exists we write $\dimtop(X)=\infty$. The topological dimension does not depend on the choice of the compatible metric $d$, and we have the inequality $\dimtop(X)\leq\HD_d(X)$ for every metric $d\in\mathcal{D}(X)$, see, eg.~\cite{HW}.

We omit the definition of the \emph{topological entropy} $h(f)$, since it is readily available~\cite{Wa} and since we will not need to work with the definition directly. Instead, we work entirely with the following two inequalities:

\begin{lemma}\label{lem:bounds} \cite{DJ, DZG, Mi}
For every metric $d\in\mathcal{D}(X)$ we have
\begin{equation*}
\HD_d(X)\logplus\Skew_d(f) \leq h(f) \leq \HD_d(X)\logplus\Lip_d(f),
\end{equation*}
except that we must ignore bounds of the form $0\cdot\infty$ or $\infty\cdot0$.
\end{lemma}

This motivates the definition of the \emph{HausLip constant} of the system $(X,f)$ as in~\eqref{start}, with the clarification that we take $0\cdot\infty=\infty\cdot0=\infty$. In other words, when we compute the HausLip constant, we ignore metrics $d\in\mathcal{D}(X)$ with respect to which $f$ is not Lipschitz continuous or $X$ has infinite Hausdorff dimension.

\begin{lemma}\label{lem:invariant}
The HausLip constant is an invariant of topological conjugacy and is bounded below by the topological entropy.
\end{lemma}
\begin{proof}
Suppose $(X,f)$ and $(Y,g)$ are conjugate by $\psi$, that is, $\psi:X\to Y$ is a homeomorphism and $g\circ\psi=\psi\circ f$. There is an induced bijection $\psi_*:\mathcal{D}(X)\to\mathcal{D}(Y)$ given by $(\psi_*d)(y,y')=d(\psi^{-1}(y),\psi^{-1}(y'))$. Once we fix metrics, our homeomorphism becomes an isometry $\psi:(X,d)\to(Y,\psi_*d)$, so that $\HD_d(X)=\HD_{\psi_*d}(Y)$ and $\Lip_d(f)=\Lip_{\psi_*d}(g)$. This implies $\HausLip(X,f)=\HausLip(Y,g)$. Finally, the inequality $\HausLip(X,f)\geq h(f)$ follows immediately from Lemma~\ref{lem:bounds}.
\end{proof}

\subsection{Known Examples} We give a short survey of some zero-dimensional and one-dimensional systems for which the entropy and the HausLip constant coincide.

\subsubsection*{Shift spaces:}\strut\\
Let $X$ be any closed, shift-invariant subspace of $\{1,\ldots,r\}^{\N_0}$, $r\geq2$, and $\sigma$ the shift transformation $(\sigma x)_n=x_{n+1}$. It's easy to show that $\HausLip(X,\sigma)=h(\sigma)$. The relevant metric is $d(x,y)=r^{-\inf\{i\mathrel{|}x_i\neq y_i\}}$. Then $d(\sigma x,\sigma y)\leq r d(x,y)$, so that $\Lip_d(\sigma)\leq r$. According to~\cite[Theorem 7.13]{Wa}, we can calculate the entropy of a shift space by counting the number of length-$n$ cylinder sets $\mathcal{B}_n$, $h(\sigma)=\lim_n \frac1n \log\#\mathcal{B}_n$. But we can also use the covering of $X$ by cylinder sets to estimate the Hausdorff dimension. Each cylinder set $B\in\mathcal{B}_n$ has diameter $r^{-n}$. If $s>h(\sigma)/\log(r)$, then $\sum_{B\in\mathcal{B}_n} |B|_d^s=(\#\mathcal{B}_n)\cdot r^{-ns} \to 0$ as $n\to\infty$, so that $\mu^{s}(X)=0$. Therefore $\HD_d(X)\leq h(\sigma)/\log r$. With our estimate for the Lipschitz constant this becomes $h(\sigma)\geq \HD_d(X)\cdot\logplus\Lip_d(\sigma)$. Together with Lemma~\ref{lem:bounds}, this shows that the metric $d$ gives equality $h(\sigma) = \HD_d(X)\cdot\logplus\Lip_d(\sigma)$.

\subsubsection*{Interval maps:}\strut\\
Let $I=[0,1]$ be the unit interval. If $f:I\to I$ is transitive and piecewise monotone, then $\HausLip(I,f)=h(f)$ by Parry's theory of constant slope maps~\cite{Pa}. Parry constructed a homeomorphism $\psi:I\to I$ and a piecewise affine map $g$ with slope $\pm\exp h(f)$ on each piece, conjugate to $f$ via $\psi$. The Euclidean metric gives the system $(I,g)$ Hausdorff dimension 1 and Lipschitz constant $\lambda$. Then the relevant metric for $f$ is just $d(x,y)=|\psi(x)-\psi(y)|$, and as in Lemma~\ref{lem:invariant} we get equality $h(f)=\HD_d(I)\cdot\logplus\Lip_d(f)$.

\subsubsection*{More interval maps:}\strut\\
Even without transitivity, if an interval map is either piecewise monotone or $C^\infty$ smooth, then a modification of Parry's construction produces conjugate interval maps not of constant slope, but with Lipschitz constants (with respect to the Euclidean metric) arbitrarily close to the exponential of the entropy~\cite{BR}. This gives equality of the HausLip constant and the entropy in these classes of interval maps as well. The paper~\cite{BR} also presents examples (non-smooth, countably many turning points) for which the infimum of Lipschitz constants among conjugate interval maps is strictly larger than the entropy. But all of those Lipschitz constants are computed with respect to the Euclidean metric. This leaves open the following question: \emph{Is there an interval map $f:I\to I$ with $\HausLip(I,f)>h(f)$?}

\subsection{Constructing Metrics}\label{subsec:metrics}

We briefly present the main tools used in this paper for constructing compatible metrics on a given compact metrizable space $X$.

\subsubsection*{Compatibility of a metric:}\strut\\
To check that a metric $d$ induces the right topology on $X$, it is enough to verify the condition: $x_n \to y$ in the topology of $X$ if and only if $d(x_n,y)\to 0$. In other words, a compatible metric is one which induces the correct notion of convergence of sequences.

\subsubsection*{Power rule:}\strut\\
If $d$ is a compatible metric on $X$ and $0<\gamma<1$, then the function $d^\gamma$ defined by $(x,y)\mapsto (d(x,y))^\gamma$ is another compatible metric on $X$. We must use exponents $\gamma<1$ in order to preserve the triangle inequality. Moreover, it follows immediately from the definitions that
\begin{equation}\label{power}
\HD_{d^\gamma}(X)=\frac{1}{\gamma}\HD_d(X) \quad \text{and} \quad \Lip_{d^\gamma}(f)=\left(\Lip_d(f)\right)^\gamma,
\end{equation}
where $f$ is any continuous map on $X$.

\subsubsection*{Product rule:}\strut\\
If $d_i$ is a compatible metric on $X_i$, $i=1,\ldots,n$, then the max metric $d(x,y)=\max_i d_i(x_i,y_i)$ is a compatible metric on the product space $\prod X_i$. Given also continuous maps $f_i$ on these spaces, the product rule for Lipschitz constants is clear from the definition, while for Hausdorff dimension there is a known inequality \cite{Fal, Ho}:
\begin{equation}\label{product}
\HD_{d}\left(\prod X_i\right)\geq\sum_{i=1}^n \HD_{d_i} X_i \quad \text{and} \quad \Lip_d(f_1 \times \cdots \times f_n) = \max_i \Lip_{d_i}(f_i).
\end{equation}
Moreover, Howroyd~\cite{Ho2} proved that the inequality for Hausdorff dimension becomes an equality if for all (or all but one) of the spaces $X_i$ the Hausdorff dimension coincides with the packing dimension (for the definition of packing dimension see, eg.,~\cite{Fal}). Notice also that an equality of Hausdorff and packing dimensions is preserved under replacement of a metric $d$ by $d^\gamma$, since the power rule~\eqref{power} holds also for packing dimension.

\subsubsection*{Frink's metrization theorem:}\strut\\
Later in the paper we will encounter an object $\rho$ which is almost a metric, but satisfies only a weak version of the triangle inequality $\rho(x,z)\leq 2\max(\rho(x,y),\rho(y,z))$. Frink~\cite{Fr} showed that this property implies the existence of a genuine metric $d$ satisfying $d(x,y)\leq \rho(x,y) \leq 4d(x,y)$.% Thus $d$ and $D$ lead to the same notion of convergence, so if $D$ was compatible with the topology then $d$ is also.

%%%%%%%%%%%%%%%%%%%%%%%%%%%%%%%%%%%%%%

\section{Linear Maps on the Torus}\label{sec:torus}

\begin{theorem}\label{th:torus}
Let $f$ be a linear map of the $n$-torus $\R^n/\Z^n$. Then the HausLip constant equals the entropy. In other words, for each $\epsilon>0$ there is a metric $d$ on $\R^n/\Z^n$ compatible with the topology such that $$\HD_d(\R^n/\Z^n)\logplus\Lip_d(f) < h(f)+\epsilon.$$
\end{theorem}

Before giving the proof, we discuss an example. Consider the system $(\R^2/\Z^2, f)$, where $f(x,y)=(2x+y,x+y)\!\!\mod\Z^2$. This map is sometimes called Arnold's cat map, and it is the prototypical example of a hyperbolic toral automorphism. It is induced by the matrix $A=\left(\begin{smallmatrix}2&1\\1&1\end{smallmatrix}\right)$ with eigenvalues $\lambda_1=\frac{3+\sqrt{5}}{2}$, $\lambda_2=\frac{3-\sqrt{5}}{2}$. Now $A$ is diagonalizable, i.e., there is a real nonsingular matrix $T$ such that $\left(\begin{smallmatrix}\lambda_1&0\\0&\lambda_2\end{smallmatrix}\right)=TAT^{-1}$. Thus, we may work with the conjugate system $(\R^2/T\Z^2,g)$, where $g(x,y)=(\lambda_1 x, \lambda_2 y)\!\!\mod T\Z^2$. Both systems are illustrated in Figure~\ref{fig:torus}. There is a very tempting choice for a metric here. Since $\R^2$ carries the Euclidean metric, we can give the torus the metric $d$ such that the projection map $\R^2\to \R^2/T\Z^2$ is a local isometry. This gives $\HD_d(\R^2/T\Z^2)=2$, $\Lip_d(g)=\lambda_1$, and $h(g)=\log\lambda_1$. But that means the metric $d$ is not very good for our purposes: the Hausdorff dimension times the logarithm of the Lipschitz constant is twice as large as the entropy.
\begin{figure}[h!!]
\includegraphics[height=4cm]{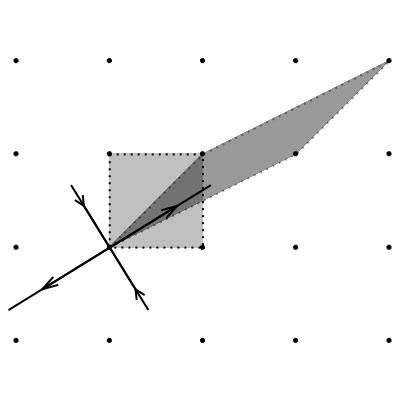} \hspace{1cm}
\includegraphics[height=4cm]{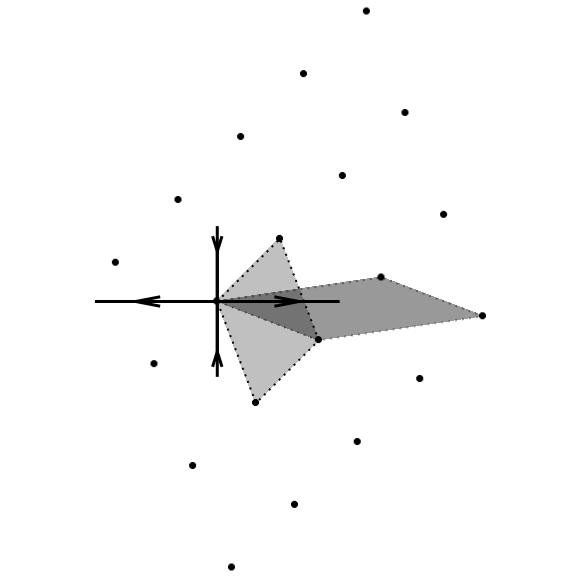}
\caption{Conjugate toral automorphisms. On the left we see $\R^2$ with the integer lattice $\Z^2$. Shown as gray boxes are a fundamental domain for the torus and its image under $A$ (which is another fundamental domain). We also see the eigendirections of $A$, which project to stable and unstable manifolds of a fixed point for the induced toral system $(\R^2/\Z^2,f)$. On the right are the corresponding objects after diagonalization, i.e. for the conjugate system $(\R^2/T\Z^2,g)$.}\label{fig:torus}
\end{figure}

This kind of construction can be repeated whenever the integer matrix $A$ is real diagonalizable, leading to a metric $d$ with $\HD_d(\R^n/T\Z^n)\cdot\logplus\Lip_d(g)$ equal to $n\log\max_i|\lambda_i|$, which doesn't correspond with the entropy~\eqref{hf}, unless all the eigenvalues have the same modulus. In order to prove our theorem, we need to construct a better metric. Let us outline our main ideas. We start by putting separate metrics on each eigendirection in such a way that the unstable directions enjoy a small uniform expansion rate. This gives us common small Lipschitz constants (for which the product rule takes a maximum) at the cost of larger Hausdorff dimensions (for which the product rule takes a sum). Putting together these metrics with the product rule, we get a metric on $\R^n$ such that the contribution of each eigendirection to $\HD\cdot\logplus\Lip$ is roughly the same as its contribution to the entropy. 

\begin{proof}[Proof of Theorem~\ref{th:torus}]
Linearity means that the map $f$ is of the form $x\mapsto Ax \!\!\mod \Z^n$ for some $n\times n$ matrix $A$ with integer entries. There is a ``real Jordan form'' $J=TAT^{-1}$ where $J,T$ are real matrices and $J$ is block diagonal
$J=\diag(J_1, \ldots, J_k)$
%\begin{equation*}
%J=\left[
%\begin{array}{ccc}
%J_1 && 0 \\
%& \ddots & \\
%0 & & J_k
%\end{array}
%\right]
%\end{equation*}
with each block $J_i$ an $n_i \times n_i$ square matrix of one of the following two forms,
\begin{equation*}
J_i=\left[
\begin{array}{ccccc}
\lambda_i & 1 \\
& \lambda_i & 1 \\
&& \lambda_i & \ddots \\
&&& \ddots & 1 \\
&&&& \lambda_i
\end{array}
\right]
\quad \text{or} \quad
J_i=\left[
\begin{array}{cccccccc}
\alpha_i & -\beta_i & 1 & 0 \\
\beta_i & \alpha_i & 0 & 1 \\
&& \alpha_i & -\beta_i & \ddots \\
&& \beta_i & \alpha_i && \ddots \\
&&&& \ddots && 1 & 0 \\
&&&&& \ddots & 0 & 1 \\
&&&&&& \alpha_i & -\beta_i \\
&&&&&& \beta_i & \alpha_i
\end{array}
\right].
\end{equation*}
The first case corresponds to a real eigenvalue $\lambda_i$; the second corresponds to a complex conjugate pair of eigenvalues $\lambda_i, \overline{\lambda_i}$, where $\lambda_i=\alpha_i+\beta_i\sqrt{-1}$. The formula for the entropy of $f$ in terms of these eigenvalues is well-known~\cite{Wa},
\begin{equation}\label{hf}
h(f)=\sum_{i=1}^k n_i \logplus |\lambda_i| = \sum_{|\lambda_i|>1} n_i\log|\lambda_i|.
\end{equation}

The change of variables $y=Tx$ induces a homeomorphism $\psi:\R^n/\Z^n\to\R^n/T\Z^n$, allowing us to replace the map $f$ with the conjugate map $g(y_1,\ldots,y_k) = (J_1y_1, \ldots, J_ky_k)\!\! \mod T\Z^n$, as summarized in the following commutative diagram
\begin{equation}\label{CD}
\begin{CD}
\R^n/\Z^n @>f>> \R^n/\Z^n\\
@V{\psi}VV @VV{\psi}V\\
\R^n/T\Z^n @>>g> \R^n/T\Z^n
\end{CD}.
\end{equation}

Now we work individually with each of the linear maps $v\mapsto J_i v$ on $\R^{n_i}$, $i=1,\ldots,k$. Fix a real number $0<\eta<1$ small enough that $1+\eta<|\lambda_i|$ whenever $|\lambda_i|>1$. We may introduce (real) vector space norms $||\cdot||_i$ by the rule
\begin{equation*}
||v||_i = ||(v_1,\ldots,v_{n_i})^{\top}||_i :=
\begin{cases}
\max\limits_{1\leq t\leq n_i} \frac{|v_t|}{\eta^t}, & \text{if }\lambda_i\in\R \\[.5em]
\max\limits_{1 \leq t\leq\frac{n_i}{2}} \frac{|v_{2t-1}+v_{2t}\sqrt{-1}|}{\eta^t}, & \text{if }\lambda_i\notin\R.
\end{cases}
\end{equation*}
The induced metric $d_i(v, v')=||v-v'||_i$ is Lipschitz equivalent to $d_{Eucl}$, the Euclidean metric. In particular, the inequalities are $\frac{d_{Eucl}(v,v')}{\sqrt{n_i}} \leq d_i(v,v') \leq \frac{d_{Eucl}(v,v')}{\eta^{n_i}}$. This shows that $d_i$ is compatible with the topology and that the dimensions are right: with respect to $d_i$ the space $\R^{n_i}$ (and every subset thereof with nonempty interior) has both Hausdorff dimension and packing dimension $n_i$.

Next we give the Lipschitz constant for the linear map $v\mapsto J_i v$ on $\R^{n_i}$,
\begin{equation*}
\Lip_{d_i}(J_i)=\begin{cases}
|\lambda_i|, & \text{if } \lambda_i\in\R \text{ and } n_i=1 \text{ or } \lambda_i\notin\R \text{ and } n_i=2,\\
|\lambda_i|+\eta, & \text{otherwise}.
\end{cases}
\end{equation*}
The first case is clear since $J_i$ represents the multiplication by $\lambda_i$ (or the multiplication by $\lambda_i$ on real and imaginary parts separately).  In the second case we use linearity and the following inequalities (with the agreement that $v_{n_i+1}=v_{n_i+2}=0$):
\begin{align*}
\text{Real } &\text{case:}\\
|| J_i v ||_i &=  \max_{1\leq t\leq n_i} \frac{ | \lambda_i v_t + v_{t+1} | }{\eta^t}\\
&\leq \max_{1\leq t\leq n_i} \left( | \lambda_i |\frac{|v_t|}{\eta^t} + \eta \cdot \frac{ |v_{t+1}|}{\eta^{t+1}} \right)
%\leq \left(|\lambda_i|+\eta\right) \max_{1\leq t\leq n_i} \frac{|v_t|}{\eta^t}
\leq \left(|\lambda_i|+\eta\right) ||v||_i \\
\text{Complex } &\text{case:} \\
|| J_i v ||_i &= \max_{1\leq t\leq\frac{n_i}{2}} \frac{|(\alpha_i + \beta_i\sqrt{-1}) (v_{2t-1}+v_{2t}\sqrt{-1}) + (v_{2t+1}+v_{2t+2}\sqrt{-1})|}{\eta^t} \\
&\leq \max_{1\leq t\leq\frac{n_i}{2}} \left( |\lambda_i| \frac{|v_{2t-1}+v_{2t}\sqrt{-1}|}{\eta^t} + \eta\frac{|v_{2t+1}+v_{2t+2}\sqrt{-1}|}{\eta^{t+1}} \right)
\leq \left(|\lambda_i|+\eta\right) ||v||_i.
\end{align*}
To see that there is no better Lipschitz constant, it suffices to choose a vector $v$ with $\frac{||J_iv||_i}{||v||_i}=|\lambda_i|+\eta$. This happens when we put $v_t = (\eta\lambda_i / |\lambda_i|)^t$, or in the complex case, $v_{2t-1}+v_{2t}\sqrt{-1}=(\eta\lambda_i / |\lambda_i|)^t$.

Next, we apply the power rule using exponents $\gamma_i$ chosen as follows:
\begin{equation*}
\gamma_i=\begin{cases}
1, & \text{if }|\lambda_i| \leq 1, \\[.5em]
\frac{\log(1+\eta)}{\log(|\lambda_i|)}, & \text{if } |\lambda_i|>1 \text{ and } \Lip_{d_i}(J_i)=|\lambda_i|, \\[.5em]
\frac{\log(1+\eta)}{\log(|\lambda_i|+\eta)}, & \text{if } |\lambda_i|>1 \text{ and } \Lip_{d_i}(J_i)=|\lambda_i|+\eta.
\end{cases}
\end{equation*}
The resulting Hausdorff dimensions and Lipschitz constants are as follows:
\begin{equation*}
\begin{array}{rcc}
\toprule
& \HD_{d_i^{\gamma_i}}(\R^{n_i}) & \Lip_{d_i^{\gamma_i}}(J_i) \\\midrule
\text{Stable \& neutral subspaces } (|\lambda_i| \leq 1): & n_i & |\lambda_i| \text{ or } |\lambda_i|+\eta \\[.5em]
\text{Unstable subspaces } (|\lambda_i|>1): & \frac{n_i\log|\lambda_i|}{\log(1+\eta)} \text{ or } \frac{n_i\log(|\lambda_i|+\eta)}{\log(1+\eta)} & 1+\eta \\ \bottomrule
\end{array}
\end{equation*}
%\begin{equation*}
%\begin{array}{rcc}
%\toprule
%& \text{Stable \& neutral subspaces} & \text{Unstable subspaces }\\
%& (|\lambda_i| \leq 1) & (|\lambda_i|>1) \\ \midrule
%\HD_{d_i^{\gamma_i}}(\R^{n_i}): & n_i & \frac{n_i\log|\lambda_i|}{\log(1+\eta)} \text{ or } \frac{n_i\log(|\lambda_i|+\eta)}{\log(1+\eta)}\\
%\Lip_{d_i^{\gamma_i}}(J_i): & |\lambda_i| \text{ or } |\lambda_i|+\eta & 1+\eta \\
%\bottomrule
%\end{array}
%\end{equation*}

Finally, we use the power rule to produce a single metric $\overline{d}$ on $\R^n$
\begin{equation*}
\overline{d}(y,y')=\max_{1\leq i\leq k} d_i^{\gamma_i}(y_i,y'_i).
\end{equation*}
By adding dimensions and taking the maximum of the Lipschitz constants we get
\begin{equation*}%\label{RnJ}
\HD_{\overline{d}}(\R^n) \cdot \logplus\Lip_{\overline{d}}(J) = \log(1+\eta) \sum_{i=1}^k \frac{n_i}{\gamma_i} \leq n\log(1+\eta) + \sum_{|\lambda_i|>1}n_i\log(|\lambda_i|+\eta).
\end{equation*}
Comparing with~\eqref{hf}, we see that this converges to $h(f)$ as we send $\eta\to0$. It remains to push our metric down to the torus. Consider the projection $\pi:\R^n\to\R^n/T\Z^n$. Let
\begin{equation}\label{dtorus}
d_g(x,y)=\min \left\{ \overline{d}(v,w) \mid v\in\pi^{-1}(x),w\in\pi^{-1}(y) \right\}.
\end{equation}
The minimum in~\eqref{dtorus} is attained because $\overline{d}$ is translation-invariant and bounded below by some constant times the Euclidean metric. As a map of metric spaces, the map $\pi$ is then a surjective local isometry. It follows that $d_g$ is compatible with the topology of the torus and that $\HD_{d_g}(\R^n/T\Z^n)=\HD_{\overline{d}}(\R^n)$. Moreover, $\Lip_{d_g}(g)\leq \Lip_{\overline{d}}(J)$ because of the inequality 
$\frac{d_g(gx,gy)}{d_g(x,y)} \leq \frac{\overline{d}(Jv,Jw)}{\overline{d}(v,w)}$,
where $v,w$ are chosen to give the minimum in~\eqref{dtorus}.

Finally, we go back through the conjugacy~\eqref{CD}, letting $d$ be the metric on the torus $\R^n/\Z^n$ such that $\psi$ is an isometry. The theorem follows by taking $\eta$ small enough.
\end{proof}

%%%%%%%%%%%%%%%%%%%%%%%%%%%%%%%%%

\section{Expansive Maps}\label{sec:expansive}

Following Walters~\cite{Wa}, we consider the following notion of expansivity for non-invertible mappings.

\begin{definition}
A continuous map $f:X\to X$ of a compact metric space $(X,d)$ is called \emph{positively expansive} if there is $c>0$ (an \emph{expansivity constant}) such that if $d(f^nx,f^ny)\leq c$ for all $n\geq0$, then $x=y$.
\end{definition}

Positive expansiveness is sometimes called one-sided expansiveness, to distinguish it from the two-sided notion usually used when discussing expansive homeomorphisms. We remark that the property of positive expansiveness is preserved if we replace the metric $d$ by another compatible metric, although the expansivity constant may change~\cite{Wa}. Our main result for positively expansive maps is as follows.

\begin{theorem}
If $f:X\to X$ is positively expansive, then $\HausLip(X,f)=h(f)$. In other words, for every $\epsilon>0$ there is a metric $d$ on $X$ compatible with its topology such that $$\HD_d(X)\logplus\Lip_d(f) < h(f)+\epsilon.$$
\end{theorem}

Fathi has a related result for expansive homeomorphisms, namely, the existence of a metric $d$ with $h(f)\geq\frac12\HD_d(X)\cdot\log k$ for a certain local anti-Lipschitz constant $k$,~\cite{Fa}. Our proof is patterned off of his -- we are grateful for his insights.

\begin{proof}
Choose a metric $\delta$ on $X$ compatible with its topology and let $c$ be the corresponding expansivity constant. Given any number $\alpha>1$ we may define
\begin{equation*}
\rho(x,y)=\alpha^{-n(x,y)}, \text{ where } n(x,y)=\inf\{ i\geq0 \mid \delta(f^ix,f^iy)>c\}.
\end{equation*}
The function $\rho$ is almost a metric: it is symmetric and $\rho(x,y)\geq0$ with equality if and only if $x=y$. Moreover, given a convergent sequence $x_j\to y$, uniform continuity gives $\rho(x_j,y)\to0$. Conversely, for each $\eta>0$ the compact set $\{(x,y)\mathrel{|}\delta(x,y)\geq\eta\}$ in the product space $X\times X$ is covered by the open sets $U_m=\{(x,y)\mathrel{|}n(x,y)\leq m\}$, $m\in\N$. Thus $\forall \eta>0$ $\exists m\in\N$ such that $\delta(x,y)\geq\eta \implies \rho(x,y)\geq \alpha^{-m}$. This proves that $\rho(x_j,y)\to0$ only if $x_j\to y$. Thus, $\rho$ satisfies all the properties of a compatible metric except the triangle inequality. Instead, we may find $m$ such that if $\delta(x,y)\geq\frac{c}{2}$, then $n(x,y)\leq m$. The triangle inequality for $\delta$ then gives us $\min(n(x,y),n(y,z)) \leq n(x,z)+m$. Requiring $\alpha>1$ to be small enough that $\alpha^m\leq2$, we get the weakened triangle inequality
\begin{equation*}
\rho(x,z)\leq 2\max(\rho(x,y),\rho(y,z)).
\end{equation*}
By Frink's metrization theorem (see Section~\ref{subsec:metrics}) there is a metric $D$ on $X$ with
\begin{equation}\label{D}
D(x,y)\leq\rho(x,y)\leq4D(x,y).
\end{equation}
Then also $D(x_j,y)\to0$ if and only if $x_j\to y$; i.e., $D$ is compatible with the topology of $X$. Finally, given any natural number $n$ we may define another topologically compatible metric $d$ by
\begin{equation}\label{d}
d(x,y)=\max_{0\leq i<n} \frac{D(f^ix,f^iy)}{L^{i/n}},
\end{equation}
where $L=\Lip_D(f^n)$. By construction, $\Lip_d(f)=L^{1/n}$. Moreover, for all $i\in\mathbb{N}$ we have $\rho(f^ix,f^iy)\leq\alpha^i\rho(x,y)$
with equality on some neighborhood $V_i$ of the diagonal in $X\times X$. Using~\eqref{D} this yields
\begin{align}\label{DLip}
D(f^ix,f^iy)&\leq4\alpha^i D(x,y), \quad \text{for }(x,y)\in X\times X, \text{ and}\\
\label{DSkew}
D(f^ix,f^iy)&\geq\tfrac14\alpha^i D(x,y), \quad \text{for }(x,y)\in V_i.
\end{align}
Combining~\eqref{d} with~\eqref{DLip} we get inequalities $D(x,y)\leq d(x,y)\leq 4\alpha^n D(x,y)$. This shows that the identity map $id:(X,d)\to(X,D)$ is a bi-Lipschitz mapping, and therefore $\HD_d(X)=\HD_D(X)$. Combining~\eqref{DLip} with~\eqref{DSkew} we get $\Lip_D(f^n)\leq16\Skew_D(f^n)$. Applying Lemma~\ref{lem:bounds} and the well-known rule $h(f^n)=nh(f)$ we get
\begin{align*}
\HD_d(X)\logplus\Lip_d(f)&=
\tfrac1n\HD_D(X)\logplus\Lip_D(f^n)\\
& \leq \tfrac1n\HD_D(X)\logplus16 + \tfrac1n\HD_D(X)\logplus\Skew_D(f^n)\\
& \leq \tfrac1n\HD_D(X)\logplus16 + h(f).
\end{align*}
To complete the proof we need only show that $\HD_D(X)<\infty$, so that for suitably large $n$ we have $\tfrac1n\HD_D(X)\logplus16<\epsilon$. Now expansive maps always have finite entropy -- this is the one-sided version of~\cite[Theorem 7.11]{Wa}. Choose $i$ such that $\alpha^i/4>1$. Then by~\eqref{DSkew} and Lemma~\ref{lem:bounds} we get
\begin{equation}\label{HDFinite}
\HD_D(X)\leq \frac{h(f^i)}{\logplus\Skew_D(f^i)}<\infty \qedhere
\end{equation}
\end{proof}

Again following Fathi, we may use our entropy inequality to recover some topological information about spaces admitting expansive mappings. These results were first proved by Ma\v{n}\'{e}~\cite{Ma} for expansive homeomorphisms.

\begin{corollary}
Any space $X$ admitting a positively expansive map $f$ has finite topological dimension. If $f$ can be chosen with zero entropy, then $X$ is totally disconnected.
\end{corollary}
\begin{proof}
As mentioned earlier, the Hausdorff dimension of a compact metric space is an upper bound for its topological dimension. Thus finite dimensionality follows from~\eqref{HDFinite}. In the case of a zero-entropy mapping,~\eqref{HDFinite} shows that $X$ has topological dimension zero, i.e., is totally disconnected.
\end{proof}

\end{document}